\documentclass[12pt]{amsart}

\usepackage{amsfonts}
\usepackage{amsmath}
\usepackage{amssymb}
\usepackage{amstext}
\usepackage{amsthm}
\usepackage{mathtools}
\usepackage{amsopn}
\usepackage{mathrsfs}
\usepackage{geometry}
\usepackage{url}
\usepackage{amsrefs}
\usepackage{hyperref}

\usepackage{booktabs} 
\usepackage{amssymb}
\usepackage{amsthm}
\usepackage{amsmath}
\usepackage{mathtools}
\usepackage{amsopn}
\usepackage{mathrsfs}
\usepackage{sansmathaccent}
\usepackage{tikz}
\usepackage{xcolor}
\usepackage{multirow}
\usepackage{booktabs}
\usepackage{xparse}
\usepackage{calc}
\usepackage{MnSymbol,wasysym}
\usepackage{pgfplots}
\usepackage{verbatim}
\usepackage{graphicx}
\usepackage{inputenc}

\hypersetup{
	colorlinks=true,
	linkcolor=red,
	filecolor=magenta,      
	urlcolor=cyan,
}

\let\emptyset\varnothing

\theoremstyle{definition}

\newtheorem{theorem}{Theorem}[section]

\newtheorem{lemma}[theorem]{Lemma}
\newtheorem{remark}[theorem]{Remark}
\newtheorem{notation}[theorem]{Notation}
\newtheorem{counter example}[theorem]{Counter Example}

\newtheorem*{Choice}{Axiom of Choice}
\newtheorem*{c-Choice}{Countable Axiom of Choice - CAC} 
\newtheorem*{d-Choice}{Axiom of Dependent Choice - ADC}
\newtheorem*{constructive-conclusion}{Constructive-Conclusion}

\usetikzlibrary{arrows}
\usepackage{pgfplots}
\usetikzlibrary{calc}
\usetikzlibrary{datavisualization.formats.functions}

\pgfplotsset{my style/.append style={axis x line=middle, axis y line=
		middle, xlabel={$x$}, ylabel={$y$}, axis equal }}

\usepackage[english]{babel}
\usepackage [autostyle, english = american]{csquotes}
\MakeOuterQuote{"}

\font\myfont=cmr12 at 6pt

\newcommand{\cc}[1]{\ignorespaces}
\allowdisplaybreaks[1]

\newcommand{\I}{\mathbf{i}}

\DeclareMathOperator\Arg{Arg}
\DeclareMathOperator\Ln{Ln}
\newcommand{\expp}{\mathbf{exp}}
\newcommand{\Lnn}{\textcolor{blue}{Ln}}
\newcommand{\Lnnn}{\textcolor{green}{Ln}}

\title{Paradox on the Countable Axiom of Choice}

\author[Babak Jabbar Nezhad]{Babak Jabbar Nezhad \\ {\myfont "Dedicated to L.E.J. Brouwer"}}
\address{Da\c{s} Maku, West Azerbaijan, Iran}
\email{babak.jab@gmail.com}


\date{}
\thanks{Babak Jabbar Nezhad has also published under the name Babak Jabarnejad~\cite{jabarnejad2016rees}}

\dedicatory{}

\begin{document}
\maketitle
\vspace{-7mm}
	
\begin{abstract}
Bishop's constructive mathematics school rejects the Law of Excluded Middle, but instead vastly makes use of weaker versions of the Choice. In this paper we pioneer an example, which shows that this road is not consistent, as our example provides a paradox. Therefore, rejecting the Law of Excluded Middle, and as an alternative using the Countable Axiom of Choice and the Axiom of Dependent Choice, still does not create a consistent structure. Actually, constructively; the Countable Axiom of Choice is an implication of the Axiom of Dependent Choice.    
\end{abstract}

\section{Introduction}
First of all, we should specify that this paper is written in the framework of Bishop's constructive mathematics. So that in entirety of this paper by constructive mathematics we mean Bishop's constructive mathematics. 

Both of the Law of Excluded Middle and the Axiom of Choice are accepted in classical mathematics but both are rejected in constructive mathematics. We are well aware that the Zermelo's Axiom of Choice revolutionized mathematics in the last century. Because of that even Bishop's school of constructive mathematicians could not be completely released from this axiom and broadly made use of weaker versions of this axiom, we mean; the Countable Axiom of Choice (CAC) and the Axiom of Dependent Choice (ADC). Specifically, rejecting the Law of Excluded Middle in constructive mathematics, led these mathematicians to apply CAC and ADC; more extensively than to what extent classical mathematicians used the Axiom of Choice itself. Therefore, in this perspective one could say that the blind scope of the classical mathematics is the Law of Excluded Middle and the blind part of the constructive mathematics is utilization of the Choice. Notice that the Axiom of Choice itself implies the Law of Excluded Middle. Therefore, the Axiom of Choice is rejected in constructive mathematics, but not weaker versions of this axiom as we explained above. Also, note that the great mathematician L.E.J. Brouwer did not mean to make any use of the Choice in his presentation of constructive mathematics. Though, this was Errett Bishop who started to interfere weaker versions of the Axiom of Choice in his presented constructive mathematics to extend results of classical mathematics to his wise interpretation of constructive mathematics.

In Bishop's constructive mathematics one rejects the Law of Excluded Middle, but they practice vastly CAC and ADC. For example using these axioms Errett Bishop builds the field of real and complex numbers and comprehensively develops one variable analysis on these fields (c.f.~\cite{BB}). The proposal is so simple; whenever one needs to use sequences in classical analysis they use proof by contradiction to prove existence of such desired sequences, which is an immediate consequence of the Law of Excluded Middle; also we may see every countably infinite set as a sequence. However, in constructive analysis one deals with sequences and they construct sequences as they are not allowed to use the Law of Excluded Middle, then they need to use Countable Choice and sometimes Dependent Choice. This leads constructive mathematicians to use either CAC or ADC. On the other hand, in constructive algebra working on fields of real, complex and algeb\textbf{}raic numbers one uses the concept of sequences to build these fields, therefore, they need constructive analysis and in the sequel they need either CAC or ADC.  In prolongation of this course of action, in Bishop's school of constructive mathematics they utilize CAC and ADC to prove that any multivariable polynomial ring over the field of algebraic numbers is a UFD, and also the field of algebraic numbers is algebraically closed; actually the last one is proved without utilization of the Choice in~\cite{R}. By the filed of algebraic numbers we mean complex numbers that are algebraic over the field of rational numbers.

What we manifest in the present paper is very short and clear by introducing a straightforward example. The example is similar to the one that we have presented in~\cite{BJN}, but this time in the framework of Bishop's constructive mathematics; with totally different arguments and different functions in used. The example takes place over the field of complex numbers but at the same time makes use of constructive properties of algebraic numbers. Roughly speaking, we introduce a nonzero differentiable function on a simply connected open set whose zeros are not isolated; certainly we introduce this example under constructive interpretation. We communicate and show that if we accept the following: some facts in constructive analysis and algebra,  the field of algebraic numbers is algebraically closed, and also a polynomial ring over the filed of algebraic numbers is a UFD, then we are led to a paradox. So that the Countable Axiom of Choice guides us to a paradox.    

Also, we should clarify that although in fuzzy theory -invented by the great scientist Aliasker Lotfi Zadeh- they keep distance from the Law of Excluded Middle, but in this theory arguments are tied to completeness of real numbers which supports this theory to make use of the concept of continuity in real analysis at least; see for example the concept of fuzzy real numbers or intersection of fuzzy sets. On the other hand, in one side real analysis; in this level; could be developed using the Law of Excluded Middle, on the other side it could be developed through making use of the Countable Choice. So that Bishop's constructive mathematics plays an important role in fuzzy theory. As a conservative estimate of our knowledge about the fuzzy theory and its current literature. In the fuzzy theory what they mean by a continuous function is not the same as Bishop's means. So that in this side there are confusing arguments in the fuzzy theory. Moreover, in the fuzzy theory they deal with many famous functions -for example in entropy- and we are able to define and construct these functions using either the Law of Excluded Middle or the Choice. At the end, in the fuzzy theory they reject the Law of Excluded Middle then they reject the following
\[
A\cap \overline{A}=\emptyset,\ A\cup\overline{A}=X,
\]
where $A$ is a set, $X$ is the universal set and the set $\overline{A}$ is the complement of $A$. But they accept the following
\[
\overline{\overline{A}}=A,\ \overline{A\cup B}=\overline{A}\cap\overline{B},\ \overline{A\cap B}=\overline{A}\cup\overline{B}.
\]
This is just ambiguity, as latter identities are implications of first identities.  

\section{A glimpse of foundation}\label{glimpse}
In this section we list some results in Bishop's constructive mathematics. Actually, in entirety of this paper by constructive mathematics we mean Bishop's constructive mathematics.

We say a set is discrete if the equality is decidable.

The Law of Excluded Middle is rejected in constructive mathematics, which asserts that $P\lor\neg P$ holds for any statement $P$, where $\neg P$ is the denial of $P$. 

\begin{Choice}
	If $S$ is a subset of $A\times B$, and for each $x$ in $A$ there exists $y$ in $B$ such that $(x,y)\in S$, then there is a function $f$ from $A$ to
	$B$ such that $(x,f(x))\in S$ for each $x$ in $A$.
\end{Choice}

\begin{c-Choice}
	This is the Axiom of Choice with $A$ being the set of positive integers.
\end{c-Choice}

\begin{d-Choice}
	Let $A$ be a nonempty set and $R$ be a subset of $A\times A$ such that for each $a$ in $A$ there is an element $a'$ in $A$ with $(a,a')\in R$. Then there is a sequence $a_0,a_1,\dots$ of elements of $A$ such that $(a_i,a_{i+1})\in R$ for each $i$.
\end{d-Choice}

The Axiom of Choice is not accepted in Bishop's constructive mathematics, as it implies the Law of Excluded Middle. CAC and ADC are accepted and widely are used in this school. Actually, the ADC implies the CAC.

\begin{remark}
Algebra is the manipulation of symbols without (necessarily) regard for their meaning. And fields in general are not purely algebraic notion. Although, a discrete field is an algebraic notion, in the sense that we do algebra in it, but a field that is not discrete is not an algebraic notion. Such as fields of real and complex numbers which are not discrete - from viewpoint of constructive mathematics -.
\end{remark}

\begin{notation}
	We denote the field of complex numbers by $\mathbb{C}$. We also denote the field of algebraic numbers by $\mathbb{C}^{\alpha}$. By the filed of algberaic numbers we mean complex numbers that are algebraic over the field of rational numbers. Finally, we denote the field of real numbers by $\mathbb{R}$. We also denote the field of real algebraic numbers by $\mathbb{R}^{\alpha}$.
\end{notation}

The open sphere of radius $r>0$ about a point $x$ in a metric space $X$ is the subset
\[
S(x,r):=\{y\in X; \rho(x,y)<r\}
\]
of $X$.

The complement of a set $S$ in $X$ is the set $X-S:=\{x\in X; \forall s\in S : x\neq s\}$. Note that the complement of an open set is a closed set.

Let $X$ be a metric space. We define the closed sphere for $r\ge 0$ to be
\[
Sc(x,r):=\{y\in X; \rho(x,y)\le r\}.
\]

Note that the closed sphere in $\mathbb{R}$ and $\mathbb{C}$ is compact.

An open set $U\subset\mathbb{C}$ is connected if any two points of $U$ can be joined by a path in $U$. We say that the open set $U$ is simply connected if it is connected and every closed path in $U$ is null-homotopic.

We don't go in detail over definitions of located set, compact set and et cetera in Bishop's constructive mathematics, as they are found in~\cite{BB}. We just mention the following

For each located set $K\subset\mathbb{C}$ and each $r>0$, we write
\[
K_r:=\{z\in\mathbb{C};\rho(z,K)\le r\}.
\] 
A totally bounded set $K\subset\mathbb{C}$ is well contained in an open set $U\subset\mathbb{C}$ if $K_r\subset U$ for some $r>0$. We then write $K\subset\subset U$. Note that if $K\subset\mathbb{C}$ is totally bounded, then $K_r$ is compact for each $r>0$.

\begin{lemma}\label{t-4}
Let $U\subset\mathbb{C}$ be open, and $f:U\rightarrow\mathbb{C}$ be uniformly continuous on each closed sphere well contained in $U$. Then $f$ is continuous on $U$.
\end{lemma}

Now, we state the following theorem from~\cite{B}
\begin{theorem}\label{t-1}
	Let $f$ be analytic and not identically zero on the connected open set $U$. Let $K$ be a compact set well contained in $U$, and $\epsilon>0$. Then either $\inf\{|f(z)|; z\in K\}>0$ or there exist finitely many points $z_1,\dots,z_n$ of $U$ and an analytic function $g$ on $U$ such that
	\[
	f(z)=(z-z_1)(z-z_2)\dots(z-z_n)g(z),\ (z\in U),
	\]
	such that $\inf\{|g(z)|; z\in K\}>0$, and $\rho(z_n,K)<\epsilon$ for each $K$.
\end{theorem}
Note that on an open set analytic condition is weaker than differentiable condition in constructive complex analysis.

We say the function $f:\mathbb{C}\rightarrow\mathbb{C}$ is an entire function if it is differentiable on $\mathbb{C}$.

Functions of
\[
\begin{array}{cc}
f:\mathbb{C}\rightarrow\mathbb{C},&f(z)=\cos(z),\\
g:\mathbb{C}\rightarrow\mathbb{C},&g(z)=\sin(z),\\
h:\mathbb{C}\rightarrow\mathbb{C},&h(z)=\expp(z),
\end{array}
\]
are entire functions, where for $z\in\mathbb{C}$ one defines
\[
\begin{array}{cc}
\expp(x+\I y)=\exp(x)(\cos(y)+\I\sin(y)),&\text{where},\ x,y\in\mathbb{R},\\
\cos(z)=\frac{\expp(\I z)+\expp(-\I z)}{2},&\sin(z)=\frac{\expp(\I z)-\expp(-\I z)}{2\I}.
\end{array}
\]

A complex number $w$ is a logarithm of a complex number $z$ if $\expp(w)=z$. Every complex number $z\neq 0$ has at least one logarithm $w$, and numbers
\[
w+2\pi\I k\ (k\in\mathbb{Z})
\]
form the totality of logarithms of $z$. If $U$ is simply connected open subset $\mathbb{C}-\{0\}$ (this is the metric complement which the same as constructive complement mentioned above), $z_0$ is any point of $U$, and $a$ is any logarithm of $z_0$, then there exists a unique differentiable function $g$ on $U$, with $g'(z)=z^{-1}$ and $g(z_0)=a$. Then we have $\expp(g(z))=z$, where $z\in U$. The function $g$ is called a branch of the logarithmic function on $U$.

If $f$ is a differentiable function from an open subset $U$ of $\mathbb{C}$ into an open subset $V$ of $\mathbb{C}$, and $g$ is differentiable on $V$, then $g\circ f$ is differentiable on $U$ if $f(K)\subset\subset V$ for each compact set $K\subset\subset U$. 

Clearly, if $g$ is an entire function and $f$ is differentiable on the open set $U$, then mentioned conditions above are automatically satisfied; so that $g\circ f$ is differentiable on $U$.

\begin{theorem}
Let $f$ be a continuous function on an open set $U\subset\mathbb{C}$, with partial derivatives $f_x$ and $f_y$ on $U$ satisfying the Cauchy-Riemann equations
\[
f_y=\I f_x.
\]
Then $f$ is differentiable on $U$, and $f'=f_x$.
\end{theorem}

\section{Some facts in constructive mathematics}

In the following we only use the Countable Axiom of Choice and the Axiom of Dependent Choice. And there is no use of the Law of Excluded Middle or any implications of it. As, all proofs are made in the framework of Bishop's constructive mathematics. We constructively prove some new results in which we did not find in the literature although we have used results in the literature of constructive mathematics to prove our results.

\begin{theorem}\label{dfta}
	Let $p(x)=a_nx^n+\dots+a_0$ be a polynomial with $a_i\in\mathbb{C}^{\alpha}$. Then we have $p(x)=a_n(x-x_1)\dots(x-x_n)$, where $x_i\in\mathbb{C}^{\alpha}$.
\end{theorem}

By discussed arguments in~\cite[Chapter 2, Section 7]{BB}, we have the following
\[
\begin{array}{cc}
(1) \ln(xy)=\ln(x)+\ln(y),& x,y\in\mathbb{R}, x> 0, y> 0,\\
(2) \ln(x/y)=\ln(x)-\ln(y),&\ x,y\in\mathbb{R}, x>0, y> 0,\\
(3) \exp(x+y)=\exp(x)\exp(y),& x,y\in\mathbb{R}.\\
(4) \exp(x)>0,& x\in\mathbb{R},\\
(5) \exp(-x)=1/\exp(x),& x\in\mathbb{R},\\
(6) \exp(x-y)=\exp(x)/\exp(y),& x,y\in\mathbb{R}.\\
(7) \exp(\ln(x))=x,& x\in\mathbb{R}, x>0.
\end{array}
\] 
Also, for every $a\in\mathbb{R}$, $a>0$, they define
\begin{gather}\label{contable-Choice-root}
a^x:=\exp(x\ln(a)),\ x\in\mathbb{R}.
\end{gather}
So that if $0<a\in\mathbb{R}$, then $\sqrt{a}>0$.

If $a+b=1$, and $a,b\ge 0$, then by~\cite[Propsition 2.6]{BB}, $(a+b)-b=a+(b-b)=a+0=a=1-b\ge 0$. Hence by~\cite[Propsition 2.11]{BB}, $b\le 1$. Similarly we have $a\le 1$. Hence we conclude that $0\le a,b\le 1$.

If $a,b\in\mathbb{R}^{\alpha}$, and $ab<0$, then $ab\neq 0$. Since $\mathbb{R}^{\alpha}$ is a discrete field, we have $a\neq 0$ and $b\neq 0$. Now, since $\mathbb{R}^{\alpha}$ is discrete we have $a>0$ or $a<0$. If $a>0$, then by~\cite[Propsition 2.11]{BB}, $a^{-1}>0$. Hence, by\cite[Propsition 2.11]{BB}, we have $aa^{-1}b<0\Rightarrow b<0$. If $a<0$, then by~\cite[Propsition 2.6, Proposition 2.11]{BB}, we have $0=a+(-a)<0+(-a)\Rightarrow-a>0$. Thus $(-a)^{-1}>0$. On the other hand, we have $ab<0\Rightarrow (-a)(-b)<0\Rightarrow -b<0\Rightarrow b>0$. Therefore, we proved the following. If $a,b\in\mathbb{R}^{\alpha}$ and $ab<0$, then either ($a>0$ and $b<0$) or ($a<0$ and $b>0$).

Suppose that $a,b\in\mathbb{R}^{\alpha}$, and $a^2+b^2=1$. Then by what we have said above we have $0\le a^2\le 1\Rightarrow a^2-1\le 0\Rightarrow (a-1)(a+1)\le 0$. Then we conclude that $(a-1)(a+1)=0$ or $(a-1)(a+1)<0$. If $(a-1)(a+1)=0$, then $a=1$ or $a=-1$. If $(a-1)(a+1)<0$, then by what we have said we have $-1<a<1$. Therefore, we have proved that if $a,b\in\mathbb{R}^{\alpha}$ and $a^2+b^2=1$, then $-1\le a,b\le 1$.

Now if we observe \cite[Chapter 2, Section 7, page 58-61]{BB}, then we see that on these pages authors define functions sine and cosine. They also construct the number $\pi$ as twice the first positive zero of the cosine function. They also prove identities $\sin(x+y)=\sin(x)\cos(y)+\cos(x)\sin(y)$, $\cos(x+y)=\cos(x)\cos(y)-\sin(x)\sin(y)$ and $\cos(x)^2+\sin(x)^2=1$, which lead us to obtain all elementary trigonometric identities related to sine, cosine and the number $\pi$. The function $f(x)=\sin{x}$ is a strictly increasing function from $(-\pi/2,\pi/2)$ onto $(-1,1)$. Since $\sin0=0$, for every $a\in\mathbb{R}^{\alpha}$, $-1<a<0$, there is a unique $-\pi/2<\theta_1<0$, such that $\sin{\theta_1}=a$, and for every $a\in\mathbb{R}^{\alpha}$, $0<a<1$, there is a unique $0<\theta_1<\pi/2$, such that $\sin{\theta_1}=a$. Moreover, for every $a\in\mathbb{R}^{\alpha}$, $-1<a<0$, there is a unique $-\pi<\theta_1<-\pi/2$, such that $\sin{\theta_1}=a$, and for every $a\in\mathbb{R}^{\alpha}$, $0<a<1$, there is a unique $\pi/2<\theta_1<\pi$, such that $\sin{\theta_1}=a$. We have a similar argument for the function $f(x)=\cos{x}$. On the other hand, we have $\cos(-\pi/2)=\cos(\pi/2)=0$, $\sin0=\sin(\pi)=0$, $\sin(-\pi/2)=-1$, $\sin(\pi/2)=1$, $\cos(\pi)=-1$, $\cos0=1$. Therefore, for every  $a\in\mathbb{R}^{\alpha}$, $-1<a<0$ (resp. $0<a<1$), there are exactly two $-\pi<\theta_1,\theta_2\le\pi$, such that $\sin{\theta_1}=\sin{\theta_2}=a$, and there are exactly two  $-\pi<\gamma_1,\gamma_2\le\pi$, such that $\cos{\gamma_1}=\cos{\gamma_2}=a$. Now, suppose that $a,b\in\mathbb{R}^{\alpha}$, and $a^2+b^2=1$. Hence $-1\le a\le 1$. We consider four different cases. Case 1, $a=0$. We have $\sin(\pi)=\sin0=a=0$. Thus $b=\pm 1$. If $b=1$, then we take $\theta=0$, and we have $\sin(\theta)=a=0$, and $\cos(\theta)=b=1$. If $b=-1$, then we take $\theta=\pi$, and we have  $\sin(\theta)=a=0$, and $\cos(\theta)=b=-1$. Case 2, $a=1$. Then $b=0$. We take $\theta=\pi/2$, and we have $\sin(\theta)=a=1$, $\cos(\theta)=b=0$.  Case 3, $a=-1$. Then $b=0$. We take $\theta=-\pi/2$, and we have $\sin(\theta)=a=-1$, $\cos(\theta)=b=0$. Case 4, $a\neq 0,1,-1$. Hence $-1<a<0$ or $0<a<1$. If $-1<a<0$, then there are exactly two $-\pi<\theta_1<-\pi/2,\quad -\pi/2<\theta_2<0$, such that $\sin(\theta_1)=\sin(\theta_2)=a$. But we have the identity $\sin^2(x)+\cos^2(x)=1$$\Rightarrow$ $\cos^2(\theta_1)+a^2=1=\cos^2(\theta_2)+a^2\Rightarrow b^2=\cos^2(\theta_1)=\cos^2(\theta_2)$. On the other hand, $-1<b<0$ or $0<b<1$. If $-1<b<0$, then we pick $-\pi<\theta_1<-\pi/2$, and we know that $-1<\cos(\theta_1)<0$. So that $b=\cos(\theta_1)$. If $0<b<1$, then we pick $-\pi/2<\theta_2<0$, and we know that $0<\cos(\theta_2)<1$. So that $b=\cos(\theta_2)$. On the other hand, $b^2=A^2$ has only two values for $b$. Then in either cases $\theta_1$ or $\theta_2$ is unique. If $0<a<1$, we have a similar argument. 

Therefore, we have the following. For every $c+\I d$, where $c,d\in\mathbb{R}^{\alpha}$ and $c^2+d^2=1$, there is a unique $\theta\in(-\pi,\pi]$ such that $c+\I d=\cos\theta+\I\sin\theta$. 

Now if $c,d\in\mathbb{R}^{\alpha}$, and they are arbitrary and at least one of them is nonzero, then we have $c+\I d=\sqrt{c^2+d^2}\left(\frac{c}{\sqrt{c^2+d^2}}+\I\frac{d}{\sqrt{c^2+d^2}}\right)$, where $-1\le \frac{c}{\sqrt{c^2+d^2}},\frac{d}{\sqrt{c^2+d^2}}\le 1$. Therefore, for each $c,d\in\mathbb{R}^{\alpha}$, we have $c+\I d=r(\cos\theta+\I\sin\theta)$, where  $\theta\in(-\pi,\pi]$ is unique, and $r\in\mathbb{R}^{\alpha}$, $r\ge 0$. 

For $c+\I d=r(\cos{\theta}+\I\sin{\theta})$, $r>0$, $-\pi<\theta\le\pi$, we denote $\theta$ as the principal argument of the algebraic number, and we write $\Arg(c+\I d)=\theta$. 

For $c+\I d=r(\cos{\theta}+\I\sin{\theta})$, $-\pi<\theta\le\pi$, $r>0$, we call $r$ the absolute value of $c+\I d$ and we denote it by $|c+\I d|$.

We state the \textit{complex conjugate root theorem} in our desired context which is valid in constructive mathematics.
\begin{theorem}
If $p(x)$ is a polynomial over the field of rational numbers, and $a+\I b$ is a root of $p(x)$, where $a,b\in\mathbb{R}$, then its conjugate $a-\I b$ is also a root of $p(x)$.
\end{theorem}
Therefore, from \textit{complex conjugate root theorem} we conclude that the absolute value of an algebraic number is also an algebraic number. Actually, we get more which says if  $z=r(\cos{\theta}+\I\sin{\theta})$ is an algebraic number, then $\cos{\theta}$ and $\sin{\theta}$ are also algebraic numbers; certainly beside $r$. 

By~\cite[Definition 2.8]{BB}, $\expp(x+\I y)=\exp(x)(\cos{y}+\I\sin{y})$, where $x,y\in\mathbb{R}$. For $z_1,z_2\in\mathbb{C}$, they prove that $\expp(z_1+z_2)=\expp(z_1)\expp(z_2)$. Note that if $z\in\mathbb{R}$, then the exp function defined on $\mathbb{R}$ coincides with the defined \textbf{exp} function on $\mathbb{C}$. Therefore,, for the algebraic number $c+\I d=r(\cos{\theta}+\I\sin{\theta})$; we are able to use the notation as $c+\I d=r\expp(\I\theta)$, where $-\pi<\theta\le\pi$. So that for two algebraic numbers $\expp(\I\theta_1),\ \expp(\I\theta_2)$ we have 
\[
\expp(\I\theta_1)\expp(\I\theta_2)=\expp(\I(\theta_1+\theta_2)),\ \expp(-\I\theta_1)=\frac{1}{\expp(\I\theta_1)}.
\] 

Now, we define the natural complex logarithmic function $\Lnn$ on $A:=\{x+\I y;\ x>0\}$ as
\[
\Lnn(z):=\ln(\sqrt{x^2+y^2})+\I\arcsin\left(\frac{y}{\sqrt{x^2+y^2}}\right).
\]
If we apply Lemma~\ref{t-4}, and using  Cauchy-Riemann equations, then we conclude that the function $\Lnn$ is differentiable on $A$. And we see that for the algebraic number $z=r\expp(\I\theta)$ we have $\Lnn(r\expp(\I\theta))=\ln(r)+\I\theta$, where $-\pi<\theta\le\pi$, where $-\pi<\theta\le\pi$ is the principal argument. And using the principal argument for two algebraic numbers $z_1$ and $z_2$ we have the following
\[
\begin{array}{cc}
\Lnn(z_1z_2)=\Lnn(z_1)+\Lnn(z_2)+2\I k\pi,&k=0,1,-1,\\
\Lnn(z_1/z_2)=\Lnn(z_1)-\Lnn(z_2)+2\I k\pi,&k=0,1,-1.
\end{array}
\]
Note that above identities remain well-defined when we work with the original function of
\[
\Lnn(z)=\ln(\sqrt{x^2+y^2})+\I\arcsin\left(\frac{y}{\sqrt{x^2+y^2}}\right).
\] 

Now, for a complex number $x+\I y$, we define the natural complex logarithmic function $\Lnnn$ on $B:=\{x+\I y;\ x<0\}$ as
\[
\Lnnn(z):=\ln(\sqrt{x^2+y^2})+\I\arccos\left(\frac{y}{\sqrt{x^2+y^2}}\right).
\] 
If we apply Lemma~\ref{t-4}, and using  Cauchy-Riemann equations, then we conclude that the function $\Lnnn$ is differentiable on $B$. And we see that for the algebraic number $z=r\expp(\I\theta)$ we have $\Lnnn(r\expp(\I\theta))=\ln(r)-\I\theta+\I\pi/2$, where $-\pi<\theta\le\pi$ is the principal argument. 

Now, we apply Theorem~\ref{t-1}, to prove the following version related to algebraic numbers.

\begin{theorem}\label{algebraic-zeros}
Let $f$ be analytic and not identically zero on the connected open set $U$. Let $K$ be a compact set well contained in $U$ and $f$ is not identically zero on $K$. Then either $\inf\{|f(z)|; z\in K\}>0$, or for only finitely many algebraic numbers $x_i\in K$ we have $f(x_i)=0$.
\end{theorem}
\begin{proof}
	Applying Theorem~\ref{t-1}, we have $f(z)=(z-z_1)(z-z_2)\dots(z-z_n)g(z)$, where $z$ is every arbitrary element in $K$. If $x_1\in\mathbb{C}^{\alpha}$, and $x_1\in K$, then $g(x_1)\neq 0$. If $f(x_1)=0$, then $(x_1-z_1)(x_1-z_2)\dots(x_1-z_n)=0$. We let $h(z):=(z-z_1)(z-z_2)\dots(z-z_n)=z^n+a_{n-1}z^{n-1}+\dots+a_0$. Now, if $f(x_1)=0$, then $h(x_1)=0$. On the other hand, we have $h(z)=(z-x_1)q(z)+p$, where 
	\begin{gather*}
	q(z)=z^{n-1}+z^{n-2}(x_1+a_{n-1})+z^{n-3}({x_1}^2+x_1a_{n-1}+a_{n-2})\\
	+\dots+z({x_1}^{n-2}+{x_1}^{n-3}a_{n-1}+\dots+a_2)\\
	+({x_1}^{n-1}+{x_1}^{n-2}a_{n-1}+\dots+x_1a_2+a_1),\\
	p={x_1}^n+{x_1}^{n-1}a_{n-1}+\dots+x_1a_1+a_0.
	\end{gather*}
	Therefore, $h(z)=(z-x_1)q(z)$. Now, if there exists another similar $x_2\neq x_1$, then $q(x_2)=0$. If this procedure stops we are done. If not, then $(z-z_1)(z-z_2)\dots(z-z_n)=(z-x_1)(z-x_2)\dots(z-x_n)$. If there is another $x_{n+1}$, then $(z-x_1)(z-x_2)\dots(z-x_n)=(z-x_1)(z-x_2)\dots(z-x_{n+1})$; this is because $\mathbb{C}^{\alpha}[z]$ is a UFD. Hence, $x_n=x_{n+1}$. This completes the proof.	
\end{proof}
\begin{remark}
Note that the proof of Theorem~\ref{algebraic-zeros}, is constructive as we have $\neg (x\neq 0)\equiv (x=0)$. So that the complement of $\mathbb{C}-\{0\}$ is the set $\{0\}$. Therefore, we may consider $f$ a function from $\mathbb{C}^{\alpha}$ onto $\mathbb{C}-\{0\}\cup\{0\}$. If the range of $f$ is $\mathbb{C}-\{0\}$, then we are done; and this function is constructive by the Countable Axiom of Choice. Otherwise, we get a contradiction; therefore, its negation is true (that is how it works in constructive mathematics) which means there is an algebraic number which goes to zero. And we proceed. 
\end{remark}

\section{The example}
Before we introduce the example we need to do some discussions as follow.

First, we fix in the entirety of this section the simply connected open set $\mathcal{D}:=\{x+\I y; -2<x<-\epsilon, -1+\epsilon<y<1-\epsilon\}$, where $\epsilon$ is a very small and positive algebraic number. Now, we consider following functions which are defined on $\mathcal{D}$, $F_1(z):=\Lnn\left(\frac{-z}{\I z+1}\right)$, $F_2(z):=\Lnn\left(\frac{\I z}{\I z+1}\right)$, and $F_3(z):=\Lnn\left(1-\sin(-\I\Ln(z))\right)$, where the function $\Ln$ is the branch of the logarithmic function in which is differentiable on the simply connected open set of $\mathcal{D}$.

Now, we discuss differentiablity of $F_1, F_2, F_3$, on $\mathcal{D}$.

\textit{(1) $F_1$.}

We have $\frac{-x-\I y}{\I (x+\I y)+1}=\frac{-x}{x^2+y^2-2y+1}+\I\frac{x^2-y+y^2}{x^2+y^2-2y+1}$. But $x^2+y^2-2y+1<4+(1-\epsilon)^2+2(1-\epsilon)+1=8-4\epsilon+\epsilon^2\Rightarrow\frac{-x}{x^2+y^2-2y+1}>\frac{\epsilon}{8-4\epsilon+\epsilon^2}$. Hence, $\frac{-z}{\I z+1}$ is well contained in the open set of $\mathcal{A}:=\{x+\I y;\ x>0\}$. Therefore, $F_1$ is differentiable on $\mathcal{D}$. 

\textit{(1) $F_2$.}

We have $\frac{\I(x+\I y)}{\I (x+\I y)+1}=\frac{x^2-y+y^2}{x^2+y^2-2y+1}+\I\frac{x}{x^2+y^2-2y+1}$. But $\frac{x^2-y+y^2}{x^2+y^2-2y+1}>\frac{\epsilon^2+\epsilon-1}{8-4\epsilon+\epsilon^2}$. Hence for $\epsilon=\frac{4}{5}$, $\frac{\I z}{\I z+1}$ is well contained in the open set of $\mathcal{A}:=\{x+\I y;\ x>0\}$. Therefore, $F_2$ is differentiable on $\mathcal{D}$. 

\textit{(1) $F_3$.}

We have $1-\sin(-\I\Ln(z))=\frac{1}{2}\frac{2x^2+2y^2+y(-x^2-y^2-1)}{x^2+y^2}+\I\frac{1}{2}\frac{x^3+xy^2-x}{x^2+y^2}$. But we have $2\epsilon^2<2x^2+2y^2$, and $-(1-\epsilon)^2-5<-x^2-y^2-1<-\epsilon^2-1$, then $(-(1-\epsilon)^2-5)(1-\epsilon)<y(-x^2-y^2-1)$. Therefore, $-6+8\epsilon-\epsilon^2+\epsilon^3<2x^2+2y^2+y(-x^2-y^2-1)$. On the other hand, we get $x^2+y^2<4+(1-\epsilon)^2$. So that if we choose $\epsilon:=\frac{4}{5}$, we see that $1-\sin(-\I\Ln(z))$ is well contained in the open set of $\mathcal{A}:=\{x+\I y;\ x>0\}$. Therefore, $F_3$ is differentiable on $\mathcal{D}$.

So that the simply connected open set of $\mathcal{D}:=\{x+\I y;\ -2<x<\frac{-4}{5},\ \frac{-1}{5}<y<\frac{1}{5}\}$ is shown in the Figure~\ref{fig: 1}, where it is the region inside red segments.
\begin{figure}
	\centering
	\begin{tikzpicture}
	\begin{axis}[my style, xtick={-2,...,1}, ytick={-1,...,1},
	xmin=-2, xmax=1, ymin=-1, ymax=1]
	\addplot [mark=non] [style=very thick, color=red]  coordinates {(-4/5,-1/5) (-4/5, 1/5)};
	\addplot [mark=non] [style=very thick, color=red]  coordinates {(-2,-1/5) (-2, 1/5)};
	\addplot [mark=non] [style=very thick, color=red]  coordinates {(-2,-1/5) (-4/5, -1/5)};
	\addplot [mark=non] [style=very thick, color=red]  coordinates {(-2,1/5) (-4/5, 1/5)};
	\end{axis}
	\end{tikzpicture}
	\caption{The open simply connected set of $\mathcal{D}$.}\label{fig: 1}
\end{figure}
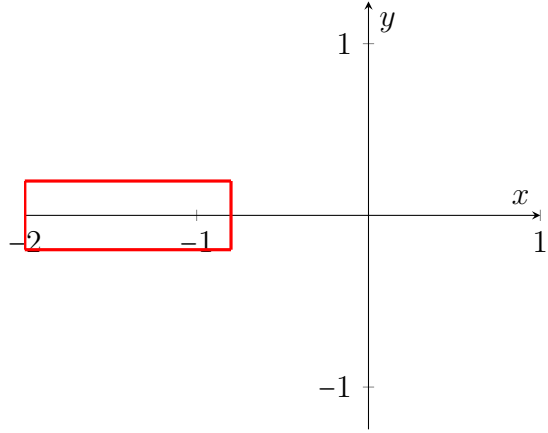

On the other hand, if the algebraic number $z:=\expp(\I\theta)$, then we have 
\begin{gather*}
\frac{-z}{\I z+1}=\frac{-\expp(\I\theta)}{\expp(\I\pi/2+\I\theta)+1}\\
=\frac{-\expp(\I\theta)}{2\cos(\theta/2+\pi/4)\expp(\I\theta/2+\I\pi/4)}\\
=\frac{-\expp(-\I\pi/4)\expp(\I\theta/2)}{2\cos(\theta/2+\pi/4)}.
\end{gather*}
And we have 
\begin{gather*}
\frac{\I z}{\I z+1}=\frac{z}{z-\I}\\
=\frac{\expp(\I\theta)}{\cos(\theta)+\I(\sin(\theta)-1)}\\
=\frac{\expp(\I\theta)}{\sin(\pi/2+\theta)+2\I(\sin(\theta/2-\pi/4)\cos(\theta/2+\pi/4))}\\
=\frac{\expp(\I\theta)}{2\sin(\pi/4+\theta/2)\cos(\pi/4+\theta/2)+2\I(\sin(\theta/2-\pi/4)\cos(\theta/2+\pi/4))}\\
=\frac{\expp(\I\theta)}{2\cos(\theta/2+\pi/4)(\cos(\theta/2-\pi/4)+\I\sin(\theta/2-\pi/4))}\\
=\frac{\expp(\I\pi/4)\expp(\I\theta/2)}{2\cos(\theta/2+\pi/4)}.
\end{gather*}
And we have
\begin{gather*}
1-\sin(-\I\Ln(z))=1-\sin(\theta)=1+\cos(\theta+\pi/2)=2\cos^2(\theta/2+\pi/4).
\end{gather*}

Now, we are ready to proceed with the example.

We consider the function
\[
\begin{array}{cc}
f:\mathcal{D}\rightarrow\mathbb{C},&f(z):=\Lnn(F_1(z))+\Lnn(F_2(z))+\Lnn(F_3(z))+\Lnnn(z)+\Ln(2)+\I\pi-\I\pi/2.
\end{array}
\]
By the argument above the function $f$ is differentiable on $\mathcal{D}$. Where $\mathcal{D}$ is an open simply connected set.  

Now, we evaluate $f$ in two set of points in the simply connected open set of $\mathcal{D}$. So that we consider two different cases as follow. 

\textit{Case (I). We consider the algebraic number $z:=\expp(\I\theta)$, $\pi/2<\theta<\pi$.}

We have
\begin{gather*}
f(z)=\Lnn(\frac{-\expp(-\I\pi/4)\expp(\I\theta/2)}{2\cos(\theta/2+\pi/4)})+\Lnn(\frac{\expp(\I\pi/4)\expp(\I\theta/2)}{2\cos(\theta/2+\pi/4)})\\
+\Lnn(2\cos^2(\theta/2+\pi/4))+\Lnnn(\expp(\I\theta))+\Lnn(2)+\I\pi-\I\pi/2\\
=\I\theta-2\Lnn(2)-2\Lnn(-\cos(\theta/2+\pi/4))+2\Lnn(-\cos(\theta/2+\pi/4))-\I\theta+2\Lnn(2)=0.
\end{gather*}

\textit{Case (II). We consider the algebraic number $z:=\expp(\I\theta)$, $-\pi<\theta<-\pi/2$.}

We have
\begin{gather*}
f(z)=\Lnn(\frac{-\expp(-\I\pi/4)\expp(\I\theta/2)}{2\cos(\theta/2+\pi/4)})+\Lnn(\frac{\expp(\I\pi/4)\expp(\I\theta/2)}{2\cos(\theta/2+\pi/4)})\\
+\Lnn(2\cos^2(\theta/2+\pi/4))+\Lnnn(\expp(\I\theta))+\Lnn(2)+\I\pi-\I\pi/2\\
=\I\theta-2\Lnn(2)-2\Lnn(\cos(\theta/2+\pi/4))+2\Lnn(\cos(\theta/2+\pi/4))-\I\theta+2\Lnn(2)+2\I\pi\\
=2\I\pi\neq 0.
\end{gather*}
Note that in the first case $\pi/4<\theta/2<\pi/2$, then  $\pi/2<\theta/2+\pi/4<3\pi/4$. Hence, we have $-\cos(\theta/2+\pi/4)>0$. And in the second case $-\pi/2<\theta/2<-\pi/4$, then  $-\pi/4<\theta/2+\pi/4<0$. Hence, we have $\cos(\theta/2+\pi/4)>0$.

Now, we consider $Sc(-1,\frac{1}{5}-\delta)\subset\subset\mathcal{D}$, where $\delta>0$ is a very small algebraic number. Clearly there are infinitely many algebraic numbers with the absolute value of 1 inside $Sc(-1,\frac{1}{5}-\delta)$; above and under the $x$-axis; examples are $\expp(-\I\pi+\frac{\I\pi}{n})$ and $\expp(\I\pi-\frac{\I\pi}{n})$ for big enough $n$'s. And on infinitely many of these algebraic numbers $f$ is zero and on infinitely many of these algebraic numbers $f$ is nonzero. Therefore, regarding Theorem~\ref{algebraic-zeros}, entire the above procedure leads us to a paradox. As the theorem says that since $f$ is analytic on the simply connected open set $\mathcal{D}$. The $Sc(-1,\frac{1}{5}-\delta)$ is a compact set well contained in $\mathcal{D}$, and $f$ is not identically zero on $Sc(-1,\frac{1}{5}-\delta)$. And $\inf\{|f(z)|; z\in Sc(-1,\frac{1}{5}-\delta) \}=0$, then there are finitely many algebraic numbers $x\in Sc(-1,\frac{1}{5}-\delta)$ such that $f(x)=0$. This is the paradox.

\end{document}